\documentclass[12pt]{amsart}
\usepackage[active]{srcltx}%%%% Inverse search
\usepackage{amsmath,amstext,amsthm,amsfonts,latexsym}
\usepackage{graphicx}
\usepackage{graphics}
\usepackage[colorlinks=true,linkcolor=red,urlcolor=black]{hyperref}
\usepackage[pagewise]{lineno}

\numberwithin{equation}{section}

\newtheorem{teo}{Theorem}[section]
\newtheorem{prop}[teo]{Proposition}
\newtheorem{defi}[teo]{Definition}
\newtheorem{lema}[teo]{Lemma}

\newtheorem{maintheorem}{Theorem}

\numberwithin{equation}{section}

\title[On mostly expanding diffeomorphisms.]{On mostly expanding diffeomorphisms.}
\date{\today}

\author[Martin Andersson]{Martin Andersson}
\address{Martin Andersson,
Departamento de Matem\'atica,
Universidade Federal Fluminense.}
\email{nilsmartin@id.uff.br}

\author[C. H. V\'asquez]{Carlos H. V\'asquez}
\address{Carlos H. V\'asquez, Instituto de Matem\'atica,
Pontificia Universidad Cat\'olica de Valpara\'{\i}so, Blanco Viel 596,
Cerro Bar\'on, Valpara\'{\i}so-Chile.} \email{carlos.vasquez@ucv.cl}
\begin{thanks}{CV  was supported by  Programa Regional Mathamsud 16-MATH-06 PHYSECO - and Proyecto Fondecyt 1130547.}
\end{thanks}


\begin{document}

\begin{abstract}
In this work we study the  class of mostly expanding partially hyperbolic diffeomorphisms. We prove that such class is $C^r$-open, $r>1$,  among the partially hyperbolic diffeomorphisms and we prove that the mostly expanding condition guarantee the existence of physical measures and provide more information about the statistics of the system.   Ma\~n\'e's classical derived-from-Anosov diffeomorphism on $\mathbb{T}^3$ belongs to this set.
\end{abstract}

\maketitle

\section{Introduction}\label{sec:intro}

%A classical result of Sinai, Ruelle, Bowen states that, associated to any uniformly hyperbolic (Axiom A) diffeomorphism $f:M \rightarrow M$, there is a finite number of probability measures $\mu_1, \ldots, \mu_\ell$ (one for each attractor) whose basins
%\[B(\mu_i) = \{ x \in M : \frac{1}{n} \sum_{k=0}^{n-1} \varphi(f^k(x)) \rightarrow \int \varphi d\mu,  \quad \forall \; \varphi\in  C^0(M; \mathbb{R}) \} \]
%all have positive Lebesgue measure. We shall call any measure with this property a physical measure (the term SRB measure is also common in the litterature). Actually, the theorem states more: the union of the basins, $B(\mu_1)\cup \ldots \cup B(\mu_\ell)$, has full Lebesgue measure in $M$. Thus, in a natural way, the 

Physical measures may be thought of as capturing the asymptotic statistical behaviour of large sets of orbits under a dynamical system. There is a strong and well known connection between the existence of physical measures and abundance, in some proper sense, of non-zero Lyapunov exponents. This is particularly true in the setting of dissipative partially hyperbolic diffeomorphisms, where non-zero central Lyapunov exponents are a crucial ingredient in nearly all of the known results (see  \cite{T2005} and \cite{B2015} for two exceptions).  During the last 15 years,  considerable effort has been done towards an understanding of the statistical behavior of partially hyperbolic diffeomorphisms with non-vanishing central Lyapunov exponents.

In the case studied by Dolgopyat \cite{D2000} and Bonatti and Viana \cite{BV2000}, they assume abundance of negative Lyapunov exponents along the center direction (mostly contracting condition). The mostly contracting condition was later shown to be $C^2$ robust, with most of its members satisfying a strong kind of statistical stability: all physical measures persist and vary continuously with small deterministic perturbations of the dynamics \cite{A2010}. Recently Dolgopyat, Viana and Yang \cite{VY2013, DVY2016} have given a detailed explanation of how bifurcations occur and given an exhaustive set of examples. They have also proved a form of continuity of the basins of physical measures.

The present work deals with the analogous but considerably harder case of diffeomorphisms whose central direction exhibits abundance of positive Lyapunov exponents. We introduce a new notion of  {\em mostly expanding} diffeomorphims (different to the one introduced in \cite{ABV2000}) and we prove they  constitute a $C^2$ open set. Moreover we show that  a mostly expanding diffeomorphims exhibit finite number of physical measures and and provide more information about the statistics of the system. Particularly, we study a notion of  ergodic stability in a non-conservative setting for such diffemorphisms.

\subsection{Mostly expanding diffeomorphims}

Let $M$ be a  closed Riemannian manifold. 
We denote by $\|\cdot\|$ the norm obtained from the Riemannian structure 
and by ${\rm Leb\:}$ the Lebesgue measure on $M$. If $V$, $W$ are normed linear spaces and $A\::\:V\to W$ is a linear map, we define
$$\|A\|=\sup\{\|Av\|/\|v\|, v\in V\setminus\{0\}\},$$
and
$${\rm m}(A)=\inf\{\|Av\|/\|v\|, v\in V\setminus\{0\}\}.$$

%A diffeomorphism $f:M\to M$  has a {\em dominated splitting} if there is
%a continuous $Df$-invariant decomposition $TM=E^{cu}\oplus
%E^{cs}$ of the tangent bundle  and constants $C\geq 0$ and $0<\la<1$ satisfying
%\begin{equation}\label{dd}
%\|Df^n|E^{cs}_{x}\|\cdot\|(Df^{-n}|E^{cu}_{f^n(x)})\|\leq C\la^n,
%\end{equation}
%for every $x\in M$, and for every $n\geq 1$. 

A diffeomorphism $f\colon M\rightarrow M$  is {\em partially hyperbolic} if there exists a  continuous $Df$-invariant splitting of $TM$, $$TM=E^s\oplus E^c\oplus E^u,$$  and there exist constants $C\geq 0$ and 

$$0<\lambda_1\leq \mu_1<\lambda_2\leq\mu_2 <\lambda_3\leq\mu_3$$ 
with $\mu_1<1<\lambda_3$ such that for all $x\in M$ and every $n\geq 1$ we have:

\begin{equation}\label{PH1}
C^{-1}\lambda_1^n\leq {\rm m}\:(Df^n(x)|E^s(x))\leq\|Df^n(x)|E^s(x)\|\leq C\mu_1^n,
\end{equation}
\begin{equation}\label{PH2}
C^{-1}\lambda_2^n\leq {\rm m}\:(Df^n(x)|E^c(x))\leq\|Df^n(x)|E^c(x)\|\leq C\mu_2^n, 
\end{equation}
\begin{equation}\label{PH3}
C^{-1}\lambda_3^n\leq {\rm m}\:(Df^n(x)|E^u(x))\leq\|Df^n(x)|E^u(x)\|\leq C\mu_3^n. 
\end{equation}
We always assume that $\dim E^\sigma\geq1$, $\sigma=s,c,u$ unless stated otherwise. We also point out that the set of $C^r$-partially hyperbolic diffeomorphisms, $r\geq 1$, is  $C^r$-open \cite[Corollary 2.17]{HP2006}. For partially hyperbolic diffeomorphisms, it is a well-known fact that there are foliations $\mathcal{F}^\sigma$ tangent to the distributions $E^\sigma$
for $\sigma=s,u$ \cite{HPS1977}. The leaf of $\mathcal{F}^{\sigma}$
containing $x$ will be called $W^{\sigma}(x)$, for $\sigma=s,u$.

An $f$-invariant probability measure $\mu$ is a {\em Gibbs $u$-state} or {\em $u$-measure} if the conditional measures of $\mu$ with respect to the partition into local strong-unstable manifolds are absolutely continuous with respect to Lebesgue measure along the corresponding local strong-unstable manifold. Section~\ref{sec:preliminaries} will be devoted to provides more properties of Gibbs $u$-states. 

\begin{defi}\label{defi:mostlyexpanding} A partially hyperbolic diffeomorphism $f:M\to M$ $f$ with  $Df$-invariant splitting $TM=E^s\oplus E^c\oplus E^u$  is {\bf mostly expanding along the central direction} if $f$ has positive central Lyapunov exponents almost everywhere with respect to every Gibbs $u$-state for $f$. 
\end{defi}

There are several notions of asymptotic expansion along the center direction in the literature and we will explain briefly the relation between mostly expanding  and the other similar conditions introduced.

\begin{defi}\label{defi:strongmostlyexpanding} A partially hyperbolic diffeomorphism $f:M \to M$ is {\bf  strongly mostly expanding along the central direction} if
\begin{equation}\label{eq:mostlyexpanding}
\lambda^c(x)=\liminf_{n\to+\infty}\frac{1}{n}\log{\rm m}(Df^n|E^{c}_x)>0
\end{equation}
for a positive Lebesgue measure set of points $x$ in every disk $D^u$ contained in a strong unstable local manifold. 
\end{defi}

The  notion above  is a mimic  of mostly contracting notion introduced in \cite{BV2000} and it is not the same as in \cite{ABV2000}, where the term mostly expanding was coined. 

We prove that if $f$ is mostly expanding in a strong sense, then it is mostly expanding   (see Proposition~\ref{prop:defequivalente1}). We do not know if the reciprocal  is true.

As we mentioned above, Alves, Bonatti and Viana \cite{ABV2000} use a different definition of mostly expanding. Their definition is more generalthan ours. In particular, they allow the strong unstable direction $E^u$ to be trivial, working with splitting of type $E^s\oplus E^{cu}$.  In our setting, we write $E^{cu}=E^{c}\oplus E^{u}$ and state  their notion of mostly expanding with a different name:

\begin{defi}\label{defi: NUEcond} We say that $f$ is {\bf non uniformly expanding} along the  center-unstable direction (for short $f$ satisfies the {\bf NUE-condition}) if there exists $c_0>0$ and $H \subset M$ of positive Lebesgue measure, such that 
\begin{equation}\label{eq:NUE}
\limsup_{n\to\infty}\frac{1}{n} \sum^{n-1}_{j=0}
\log\|Df^{-1}|E^{cu}_{f^j(x)}\|\leq -c_0<0.
\end{equation}
holds for every $x \in H$. 
\end{defi}

Recently, condition (\ref{eq:NUE}) was weakened by Alves, Dias, Luzzatto and Pinheiro in \cite{ADLP2014}.  

\begin{defi}\label{defi: wNUEcond} A partially hyperbolic diffeomorphism $f$ as above is {\bf weakly non-uniformly expanding} along the center-unstable direction (or it satisfies the {\bf wNUE-condition} for short), if there exists $c_0>0$ and $H \subset M$ of positive Lebesgue measure, such that 
\begin{equation}\label{eq:wNUE}
\liminf_{n\to\infty}\frac{1}{n} \sum^{n-1}_{j=0}
\log\|Df^{-1}|E^{cu}_{f^j(x)}\|\leq -c_0<0.
\end{equation}
\end{defi}

We remark that the $\liminf$ in  \eqref{eq:wNUE} implies that the growth only needs to be verified on a sub sequence of iterates, in contrast to the $\limsup$ in \eqref{eq:NUE}, where the condition needs to be verified for all sufficiently large times.

Our first result reveals the motivation of this work and the reason to introduce a new notion of mostly expanding: Properties \eqref{eq:NUE} and \eqref{eq:wNUE}  are not robust. 

\begin{maintheorem}\label{mteo:A}
Satisfying the NUE-condition (or wNUE-condition) on a set $H\subseteq \mathbb{T}^3$ with full Lebesgue measure is not a robust property among the set of partially hyperbolic $C^r$-diffeomorphisms on $\mathbb{T}^3$, $r>1$. 
\end{maintheorem}

The statement above holds also if we replace `full measure' by `positive measure'. In fact, in  Section~\ref{sec:DHP Blocks}), we exhibits examples where  each property in the statement above is not robust. The examples are inspired by a construction due to  Dolgopyat, Hu and Pesin \cite[Appendix B]{BP1999}. They  provide an example of a non-uniformly hyperbolic volume preserving diffeomorphism on $\mathbb{T}^3$ with countably many ergodic components.

Even though our notion is more restrictive than NUE-condition or wNUE-condition (see Section~\ref{sec:demteosCD} for details), to be mostly expanding is a robust property  so mostly expanding diffeomorphisms are a good setting to looking for robust statistical properties.

\begin{maintheorem}\label{mteo:B}
The class of mostly expanding partially hyperbolic diffeomorphisms constitutes a $C^r$-open subset of ${\rm Diff}^r(M)$, $r>1$.
\end{maintheorem}

%Since our definition  is inspired from  mostly contracting diffeomorphisms, mostly expanding diffeomorphism  keeps  interesting properties of the last one. For instance,  it was shown in \cite{A2010} that mostly contracting is robust, i.e. $C^r$ open for $r>1$. 

\subsection{Mostly expanding condition and existence of physical measures}
Physical measures may be thought of as capturing the asymptotic statistical behavior of large sets of orbits under a dynamical system. Recall that if $\mu$ is an $f$-invariant measure,  the {\em basin} of $\mu$ is the set

$${\mathcal B}(\mu)=\{z\in M\::\: \lim_{n\to\infty}\frac{1}{n}\sum_{k=0}^{n-1} \varphi(f^{k}(z))=\int_M \varphi\:d\:\mu, \mbox{ for all }\varphi\in C^0(M,\mathbb{R})\}.$$

It is well known that if $\mu$ is ergodic, then ${\mathcal B}(\mu)$ has full $\mu$-measure. The measure $\mu$  is {\em physical} or {\em SRB measure}, if 
${\rm Leb\:}({\mathcal B}(\mu))>0$. 

There is a strong and well known connection between the existence of physical measures and abundance, in some proper sense, of non-zero Lyapunov exponents. This is particularly true in the setting of dissipative partially hyperbolic diffeomorphisms, where the asymptotic expansion (or contraction) on the central subbundle is a crucial ingredient in nearly all of the known results. For instance, Alves, Bonatti and Viana \cite{ABV2000} as well as Alves, Dias, Luzzatto and Pinheiro  \cite{ADLP2014} showed that if $f$ is a $C^r$- partially hyperbolic diffeomorphism satisfying the NUE-condition (resp. wNUE-condition) for $H=M$, then it  exhibits finitely many physical measures and the union of their basins covers a full Lebesgue measure subset of $M$. Nevertheless, the techniques and methods that  used by  Alves, Dias, Luzzatto and Pinheiro in \cite{ADLP2014} to deal with this weaker assumptions are completely different from those used in Alves Bonatti and Viana in \cite{ABV2000}.

In \cite{ABV2000}, Alves, Bonatti and Viana ask if is it possible to conclude the existence of physical measures if the non-uniform expansion condition \eqref{eq:NUE} is replaced by
condition \eqref{eq:mostlyexpanding}. Our  next theorem gives essentially a positive answer to such question, although instead of requiring that the condition hold on a positive Lebesgue measure set, we require it to hold  on a positive leaf volume subset of every unstable disk.

\begin{maintheorem}\label{mteo:C}
If $f$ is a  mostly expanding partially hyperbolic diffeomorphism, then $f$ has a finite number of physical measures whose basins together cover Lebesgue almost every point in $M$.
\end{maintheorem}

We point out, however, that there is no possibility of having the phenomenon of intermingled basins of attraction in the mostly expanding case (see Lemma~\ref{le:openbasin} and also \cite{UV2015}).

\subsection{Mostly expanding and stable ergodicity}

We now consider the question of uniqueness of physical measures, not just for $f$, but also for its small perturbations. This is related to the stable ergodicity problem in a dissipative setting studied in \cite{BDP2002, BDPP2008, A2010} for  mostly contracting diffeomorphisms  and by \cite{V2009} in the case of mostly expanding diffeomorphisms. Our examples in section 5 show that, in contrast to the mostly contracting case, where it was shown in \cite{A2010} that uniqueness of the physical measure implies robust uniqueness of the physical measure, having the NUE or wNUE satisfied on a set of full Lebesgue measure does not imply that uniqueness of the physical measure is a robust property. We refer the reader to \cite{P2010, PC2010} and the references therein for a more exhaustive information. 

Recall that a  foliation is minimal if its leaves are dense. The strongly stable foliation $\mathcal{F}^{s}(f)$ of a partially hyperbolic diffeomorphism $f:M\to M$ is $C^r$-robustly minimal if there exists a $C^r$-neighborhood $\mathcal{U}$ of $f$ such that $\mathcal{F}^{s}(g)$ is minimal for every $g\in\mathcal{U}$. 

\begin{maintheorem}\label{mteo:D}
Assume that  $f$ is a mostly expanding partially hyperbolic $C^r$-diffeomorphism, $r>1$.  Suppose that the strongly stable (resp. unstable) foliation $\mathcal{F}^{s}(f)$ $f$ is $C^r$-robustly minimal. Then any $C^r$ diffeomorphism $g$ close enough to $f$ in the $C^r$ topology has a unique physical measure $\mu_g$ whose  basin $\mathcal{B}(\mu_g)$ has full volume in whole manifold $M$. 
\end{maintheorem}

The same result above holds if we replace the hypothesis of strongly stable (resp. unstable) foliation by robust transitivity. Conditions under which one of the strong foliations of a partially hyperbolic diffeomorphism is  robustly minimal was provides by Pujals and Sambarino \cite{PS2006} and Bonatti, Diaz and Ures \cite{BDU2002} and Nobili \cite{N2015}. 

The work is organized as follows. Section~\ref{sec:preliminaries} is devoted to recall  known results  which will be used later. In Section~\ref{sec:demteoB} we prove  Theorem~\ref{mteo:B}. Theorem~\ref{mteo:C} and Theorem~\ref{mteo:D} are proved in Section~\ref{sec:demteosCD}. As we mention before, in Section~\ref{sec:DHP Blocks} we show examples where  NUE-condition and wNUE-condition  fails to be robust,  proving the statement of Theorem~\ref{mteo:A}. Finally, in Section~\ref{sec:DA}, and following the ideas developed in \cite{BV2000}, we shown that classical example of a non-hyperbolic robustly transitive partial hyperbolic diffeomorphism due to Ma\~{n}e \cite{M1978} is mostly expanding along the central direction. In particular the previous result can be applied to such class of examples.

\section{Preliminaries}\label{sec:preliminaries}

Along this section we summarize the main results related to the existence of physical measure in the setting of partially hyperbolic diffeomorphisms. The key ingredient are the Gibbs $u$-states.

Along this section $f$ be a partial hyperbolic diffeomorphism. We denote by $\mathcal{M}(M)$  the set of  probability measures defined on $M$ provided with the weak* topology and  denote by $\mathcal{M}(f)$  the set of  probability measures invariant by $f$. It is well known that  $\mathcal{M}(f)$ is a convex compact subset of $\mathcal{M}(M)$ and moreover, every invariant measure have a decomposition into ergodic measures (cf. ~\cite{M1987}).

An $f$-invariant probability measure $\mu$ is a {\em Gibbs $u$-state} or {\em $u$-measure} if the conditional measures of $\mu$ with respect to the partition into local strong-unstable manifolds are absolutely continuous with respect to Lebesgue measure along the corresponding local strong-unstable manifold. More precisely, given a point $z\in M$, we define a {\em foliated box} of $z$ in the following way. Pick a strong-unstable disk  $D$ with center at $z$, and take a cross section $\Sigma$ to the strong-unstable foliation through the center point $z$.  Then there exists $\phi\::\: D\times \Sigma\to M$ which is a homeomorphism onto its image, such that  $\phi$ maps each horizontal $D\times \{y\}$ diffeomorphically to an unstable domain through $y$. We may choose $\phi$ such that $\phi(z,y)=y$ for all $y\in \Sigma$ and $\phi(x,z)=x$ for all $x\in D$. In what follows, we identify $D\times \Sigma$, 
and each $D\times\{y\}$, with their images under this chart $\phi$. Given any measure 
$\xi$ on $D\times \Sigma$, we denote by $\hat{\xi}$ the measure on $\Sigma$ defined by
\begin{equation}\label{eq:medidacompl}
 \hat{\xi}(B)=\xi(D\times B).
\end{equation}
 An $f$-invariant probability measure $\mu$ is a {\em Gibbs $u$-state} or {\em $u$-measure} if, for every foliated neighborhood $D\times \Sigma$ such that $\mu(D\times \Sigma)>0$, the conditional measures of $\mu|(D\times \Sigma)$ with respect to the partition into  strong unstable plaques $\{ D\times \{y\}\::\: y\in \Sigma\}$ are absolutely continuous with respect to Lebesgue measure along the corresponding plaque.

Gibbs $u$-states play a key role in the theory: If $\mu$ is a physical measure for a partially hyperbolic diffeomorphism, then  $\mu$ must be a Gibbs $u$-state \cite[Section 11.2.3]{BDV2005}.

In the early eighties, Pesin and Sinai \cite{PS1982} proved that the set of the Gibbs $u$-states of $f$ is non-empty for $C^r$- partially hyperbolic diffeomorphisms, $r>1$. More precisely, denote by $u$ the dimension of the bundle $E^{u}$. If $D^u$ is a $u$-dimensional disk inside a strong unstable leaf, and ${\rm Leb\:}_{D^u}$ denotes the volume measure induced on $D^u$ by the restriction of the Riemannian metric to $D^u$, then every accumulation point of the sequence of probability measures
$$\mu_n=\frac1n\sum_{j=0}^{n-1}f_*^j\left(\frac{{\rm Leb}_{D^u}}{{\rm Leb}_{D^u}(D^u)}\right)$$
is a  Gibbs $u$-state with densities with respect to Lebesgue measure along the strong unstable plaques uniformly bounded away from zero and infinity.  It is possible to extend the result obtained by Pesin and Sinai \cite[Theorem 11.16]{BDV2005}:

\begin{prop}\label{prop:uM6}
There exists $E\subseteq M$ intersecting every unstable disk on a full Lebesgue measure subset, such that for any $x\in E$, every accumulation point $\nu$ of 
\begin{equation}\label{eq:uM6}
 \nu_{n,x}=\frac1{n}\sum_{j=0}^{n-1}\delta_{f^j(x)}
\end{equation}
is a Gibbs $u$-state.
\end{prop}

The support of any Gibbs $u$-state consists of entire strong unstable leaves \cite[Corollary 11.14]{BDV2005}. Moreover, a convex combination of Gibbs $u$-states is a Gibbs $u$-state. 
Conversely,   if  $\mu$ is a Gibbs $u$-state, its ergodic components are  Gibbs $u$-states whose  densities are uniformly bounded away from zero and infinity \cite[Lemma 11.13]{BDV2005}. 

Denote by $\mathcal{G}^u(f)\subseteq \mathcal{M}(f)$ the set of Gibbs $u$-states  of $f$. The assertion above implies that $\mathcal{G}^u(f)$ is convex. Furthermore,  the set $\mathcal{G}^u(f)$ provided with the weak* topology  is closed in  $\mathcal{M}(f)$ and so compact \cite[Theorem 5]{BDPP2008}. Moreover, given any sufficiently small $C^r$-neighborhood $\mathcal{U}$ of $f$, $r>1$, the set 
\begin{equation}
\mathcal{G}^u(\mathcal{U},M) = \{(g,\mu) : g \in \mathcal{U} \text{ and }\mu \text{ a Gibbs $u$-state of  $g$ }\}
\end{equation}
is closed in $\mathcal{U} \times \mathcal{M}(M)$ \cite[Remark 11.15]{BDV2005} when we consider the product topology.

In the partial hyperbolic setting, notions like physical measures, Gibbs $u$-states, non zero Lyapunov exponents, and  stable ergodicty are closely related. We will explain some known relations useful for ours purposes. We refer the reader to \cite{BDV2005, P2010,PC2010,Y2002}, and then  references therein, for a complete discussion about the relations to be discussed now. 

As mentioned above, physical measures are Gibbs $u$-states in the setting of partially hyperbolic diffeomorphisms, but the converse it is not true even in the uniformly hyperbolic setting  as the reader can notice  from the  example at the end of this section. It is well known that if $\mu$ is an ergodic Gibbs $u$-state with negative central Lyapunov exponents, that is, if
$$\limsup_{n\to+\infty}\frac{1}{n}\log\|Df^n|E^{c}_x\|<0$$
for $\mu$-almost every point $x\in M$, then $\mu$ is a physical measure (see \cite[Theorem 3]{Y2002} and the reference therein). The statement follows from classical arguments \cite{PS1989}.

This is a good motivation to introduce the notion of mostly contracting: A partially hyperbolic diffeomorphism $f$  is {\em mostly contracting} along the center subbundle if
\begin{equation}\label{eq:mostlycontracting}
\limsup_{n\to+\infty}\frac{1}{n}\log\|Df^n|E^{c}_x\|<0
\end{equation}
for a positive Lebesgue measure set of points $x$ in every disk $D^u$ contained in a strong unstable local manifold. In such case, it was proved in \cite{BV2000} that if $f$  is partially hyperbolic and strongly mostly contracting along the center subbundle, then $f$ admits finitely many ergodic physical  measures, and the union of their basins covers a full Lebesgue measure subset of the basin of $M$. A related notion of mostly contracting was studied by Dolgopyat in \cite{D2000}.  Ten years later, in \cite{A2010} was proved that  mostly contracting  property \eqref{eq:mostlycontracting} is equivalent to 
every (ergodic) Gibbs $u$-state  has negative central Lyapunovs exponents.

In the setting of partially hyperbolic diffeomorphisms with absence of strong unstable direction, that is, when  $f:M\to M$ is a $C^r$-diffeomorphism, $r>1$, with decomposition of the tangent bundle $TM=E^s\oplus E^c$, it makes no sense to speak of Gibbs $u$-states. For this setting,   Alves, Bonatti and Viana \cite{ABV2000} introduced the notion of Gibss $cu$-state using the fact that, in the presence of positive Lyapunov exponents, there are Pesin invariant unstable manifolds. Thus Gibss $cu$-states correspond to a non-uniform version of  Gibbs $u$-states: An invariant probability measure $\mu$ is a {\em Gibbs $cu$-state} if the $m$ largest Lyapunov exponents are positive $\mu$-almost everywhere, where $m=\dim E^c$, and the conditional measures of $\mu$ along the corresponding local unstable Pesin's manifolds are almost everywhere absolutely continuous with respect to Lebesgue measure on these manifolds.

Alves, Bonatti and Viana showed that  a $C^r$-non uniformly expanding partially hyperbolic diffeomorphism (recall Definition~\ref{defi: NUEcond}), $r>1$,  exhibits (ergodic) Gibbs $cu$-states which are physical measures \cite{ABV2000}. Moreover, if $H=M$, then $f$ admits finitely 
many (ergodic) physical measures, and the union of their basins covers a full Lebesgue measure subset of $M$. Same conclusion was reached by Alves, Dias, Luzzatto and Pinheiro in \cite{ADLP2014} under the wNUE-condition (recall Definition~\ref{defi: wNUEcond}).
We  record the precise statement obtained by the authors above for future references.

\begin{prop}[\cite{ADLP2014}, Theorem A]\label{prop:ADLP2014}
Let $f : M \to M$ be a $C^r$, $r>1$, partially hyperbolic diffeomorphism with decomposition $TM = E^s \oplus E^{cu}$. Assume that there exists a subset $H\subseteq M$ of positive Lebesgue measure on which $f$ is  weakly non-uniformly expanding along $E^{cu}$. Then
\begin{enumerate}
\item there exist closed invariant transitive sets $\Omega_1, \dots, \Omega_\ell$ such that for Lebesgue almost every $x \in H$ we have $\omega(x) = \Omega_j$ for some $1 \leq j \leq \ell$;
\item  there exist (ergodic) Gibbs $cu$-states $\mu_1, ..., \mu_\ell$ supported on the sets $\Omega_1, \dots, \Omega_\ell$, whose basins whose basins are open up to a zero Lebesgue measure set, such that for Lebesgue almost every $x \in H$ we have $x\in \mathcal{B}(\mu_j)$ for some $1 \leq j \leq \ell$.
\end{enumerate}
\end{prop}

In \cite{V2007}, the author establish  several properties of Gibbs $cu$-states which are similar to the ones for Gibbs $u$-states: \cite[Theorem 2.1]{V2007}, The ergodic components of a  Gibbs $cu$-state $\mu$ are Gibbs $cu$-states whose  densities are uniformly bounded away from zero and infinity. Conversely, a convex combination of Gibbs $cu$-states is a Gibbs $cu$-state. The support of any Gibbs $cu$-state consists of entire center-unstable leaves. In the setting of partially hyperbolic diffeomorphisms with non-uniform expansion,  every ergodic physical measure is a  Gibbs $cu$-state.

In the case  where $f$ is a non-uniformly expanding partially hyperbolic diffeomorphism with decomposition of the tangent bundle $TM=E^s\oplus E^c\oplus E ^u$, every Gibbs $cu$-state is in fact a Gibbs $u$-state with positive central Lyapunov exponents. The converse is not true, even in the case of Anosov diffeomorphisms. 

\noindent{\bf Example:} Consider two linear Anosov diffeomorphisms $A_1, A_2$ over the torus $\mathbb{T}^2$ with splittings $E^u_1\oplus E^s_1$ and $E^u_2\oplus E^s_2$ respectively, and with (unstable) eigenvalues $\lambda_1>\lambda_2>1$ respectively. Now consider $f\::\:\mathbb{T}^2\times\mathbb{T}^2\to \mathbb{T}^2\times\mathbb{T}^2$ defined by $f=A_1\times A_2$. If we consider the decomposition $E^u=E^u_1$, $E^c=E^u_2$ and $E^s=E^s_1\oplus E^s_2$ of the tangent bundle of $\mathbb{T}^2\times\mathbb{T}^2$, then $f$ is a partially hyperbolic diffeomorphism with positive central Lyapunov exponent at every point. Consider the measure $\mu=\mu_1\times\mu_2$, where $\mu_1$ is an Gibbs $u$-state for $A_1$ and $\mu_2$ is the Dirac measure supported on a periodic orbit of $A_2$. Then $\mu$ is a  Gibbs $u$-state for $f$ whose central Lyapunov exponents are positive but it is not a Gibbs $cu$-state.

%%%%%%%%%%%%%%%%%%%%%%%%%%%%%%%%%%%%%%%%%%%%%%%%%%%%%%%%%%%%%%%
%%%%%%%%%%%%%%%%% DEMOSTRACION TEOREMA A %%%%%%%%%%%%%%%%%%%%%%
%%%%%%%%%%%%%%%%%%%%%%%%%%%%%%%%%%%%%%%%%%%%%%%%%%%%%%%%%%%%%%%

\section[Proof of Theorem~\ref{mteo:B}]{Proof of Theorem~\ref{mteo:B}}\label{sec:demteoB}

For every $f\in {\rm Diff}^r(M)$ partially hyperbolic and   $x\in M$,  denote by 
\begin{equation}\label{eq:centralLyapexp}
\lambda^c(x,f):=\liminf_{n\to+\infty}\frac{1}{n}\log{\rm m}(Df^n|E^{c}_x).
\end{equation}
The diffeomorphism $f$ is   {\em  strongly mostly expanding} (cf. Definition~\ref{defi:strongmostlyexpanding}) if $\lambda^c(x,f)>0$ for a positive Lebesgue measure set of points $x$ in every disk $D^u$ contained in a strong unstable local manifold. If $x\in M$ is a regular point, the number above is the minimum of the  Lyapunov  exponents of $x$  whose Osedelec splitting is contained in the central direction. According to Osedelec Theorem (see \cite[Theorem 10.1]{M1987}), the set of regular points is a Borel set with total measure.
In particular, the set of regular points has full measure with respect to every Gibbs $u$-state .  Note that the function defined by (\ref{eq:centralLyapexp}) is $f$ invariant. Then, $\lambda^c(x,f)=:\lambda^c(\mu,f)$ is constant, for $\mu$-a.e. $x$ when $\mu$ is ergodic.   

\begin{prop}\label{prop:defequivalente1}
Let $f:M\to M$ be a strongly mostly expanding partially hyperbolic diffeomorphism, then $f$ is mostly expanding.
\end{prop}

\proof Assume that $f$ is strongly mostly expanding. It is enough to prove the assertion for ergodic Gibbs $u$-states and, in such case, to prove that all central Lyapunov exponents are positive on a set of positive $\mu$-measure. Pick a  foliated neighborhood $D\times \Sigma$ such that $\mu(D\times \Sigma)>0$. By definition of Gibbs $u$-states, the conditional measures of $\mu|(D\times \Sigma)$ with respect to the partition into  strong unstable discs $\{ D\times \{y\}\::\: y\in \Sigma\}$ are absolutely continuous with respect to Lebesgue measure along the corresponding unstable discs. From such disintegration and absolutely continuity, we conclude for $\hat{\xi}$-almost every point  $y\in\Sigma$ (recall that $\hat{\xi}$ was defined in \eqref{eq:medidacompl}), the central Lyapunov exponents are well defined in a set of Lebesgue full measure in each the strong unstable discs $D\times \{y\}$. Since $f$ is strongly  mostly expanding, such central Lyapunov exponents must be positive. 
\endproof

Recall that  $\mathcal{M}(M)$ denotes the set of probability measures on $M$ provided with the weak* topology and $\mathcal{G}^u(f)\subseteq \mathcal{M}(M)$ denotes the set of Gibbs $u$-states  of $f$. Let $\mathcal{S}\subseteq {\rm Diff}^r(M)\times \mathcal{M}(M)$ be the set of pairs $(f,\mu)$ where $f$ is  a $C^r$ partially hyperbolic diffeomorphism and $\mu$ is a Gibbs $u$-state for $f$.  We consider $\mathcal{S}$ endorsed with the product topology induced from ${\rm Diff}^r(M)\times \mathcal{M}(M)$.

Let us consider the function $\Lambda^c\::\:\mathcal{S}\to\mathbb{R}$ defined as
$$\Lambda^c(f,\mu):=\int_M {\lambda}^c(x,f)\:d\mu(x).$$

Now,  let $\mu$ be any Gibbs $u$-state. By convexity of the set of Gibbs $u$-states of $f$, $\mu$ is a convex combination of ergodic Gibbs $u$-states $(\mu_x)_{x\in M}$. Therefore, the ergodic cecomposition theorem implies that
\begin{equation}\label{eq:Lambdac}
\Lambda^c(f,\mu):=\int\lambda^c(x,f)d\mu(x)=
\int\int\lambda^c(x,f)d\mu_x\:d\mu(x)=\int\lambda^c(\mu_x)\:d\mu(x).
\end{equation}

The next lemma will be useful in the proof  of Proposition~\ref{prop:meaberto}.

\begin{lema}\label{le:equivpositivo} A partial hyperbolic diffeomorphims $f$ is  mostly expanding if and only if $\Lambda^c(f,\mu)>0$, for every $\mu\in\mathcal{G}^u(f)$.
\end{lema}
\proof Assume that $f$ is mostly expanding. If $\nu$ is an ergodic Gibbs $u$-state, 
$$\lambda^c(x,f):=\lambda^c(\nu,f)>0,$$
for $\nu$-almost every point $x\in M$. Then 
$$\Lambda^c(f,\nu)=\int_M {\lambda}^c(x,f)\:d\nu(x)=\lambda^c(\nu,f)>0.$$
If $\nu$ is not ergodic, then from (\ref{eq:Lambdac}) we have
$$\Lambda^c(f,\nu):=\int\lambda^c(x,f)d\nu(x)=
\int\int\lambda^c(x,f)d\nu_x\:d\nu(x)=\int\lambda^c(\nu_x,f)\:d\nu(x)>0$$
since every $\nu_x$ is an ergodic Gibbs $u$-state and then $\lambda^c(\nu_x,f)>0$.

Now,  we assume that $\Lambda^c(f,\mu)>0$, for every $\mu\in \mathcal{G}^u(f)$. If $\mu$ is ergodic, then follows from  $\Lambda^c(f,\mu)=\lambda^c(\mu,f)$ that $\mu$ has positive central Lyapunov exponents. If $\mu$ is non ergodic and $\lambda^c(x,f)\leq 0$  for every $x\in A$, where $A$ is a set of positive $\mu$-measure. Then by Ergodic Decomposition Theorem, there exists an ergodic component $\mu_x$ of $\mu$ such that $x\in A$, $\mu_x(A)>0$, and 
$\lambda^c(x,f)=\lambda^c(\mu_x,f)\leq 0$, which is a contradiction.
\endproof

For every  $n\geq 1$, define $L_n:\mathcal{S}\to \mathbb{R}$ by
\begin{equation}
L_n(f,\mu):=\int\log{\rm m}(Df^{n}|E^c_{x})\: d\mu(x)
\end{equation}
Recall that the set $\mathcal{S}$ is endowed with the product topology induced from  ${\rm Diff}^r(M)\times \mathcal{M}(M)$. Therefore,  $L_n:\mathcal{S} \to \mathbb{R}$ is continuous, for every $n\geq 1$. Moreover, for a fixed  $(f,\mu)\in \mathcal{S}$, the sequence $(L_n(f,\mu))_{n\geq 1}$ is super additive, that is, for every integers $n,m\geq 1$, we have 
\begin{equation}\label{eq:supaddL}
L_{n+m}(f,\mu) \geq L_n(f,\mu)+L_m(f,\mu).
\end{equation}
This implies that the following limit exists (or is equal to $+\infty$).
$$\alpha=\alpha(f,\mu)=\lim_{n\to+\infty}\frac1n L_n(f,\mu)=\sup\{\frac{1}{n}L_n(f,\mu)\::\: n\in\mathbb{N}\}.$$

\begin{lema}\label{le:Lambdalim} For every $(f,\mu)\in \mathcal{S}$,
\begin{equation}\label{eq:Lambdalim}
\Lambda^c(f,\mu)=\lim_{n\to+\infty}\frac1n L_n(f,\mu).
\end{equation}
\end{lema}
\proof This lemma is a direct consequence of dominated converge theorem. In fact, fix $(f,\mu)\in \mathcal{S}$ and let us consider the sequence of  
$\mu$-integrable functions on $M$ defined by  $\Psi_n(x)=\frac{1}{n}\log{\rm m}(Df^n|E^c_x)$, $x\in M$, $n\geq 1$. Then $\Psi_n$ converge $\mu$-almost every point to $\lambda^c(\cdot,f)$ and by partial hyperbolicity, $\Psi_n(x)\leq \frac{1}{n}\log C^{-1}+\log\lambda_3$, where $C\geq 0$ and $\lambda_3> 1$ are the constants in \eqref{PH3}. Hence,
\begin{eqnarray*}
\Lambda^c(f,\mu)&=&\int\lambda^c(x,f)d\mu(x)=\lim_{n\to+\infty}\int\Psi_n(x)\\&=&\lim_{n\to\infty}\frac1n\int\log  {\rm m}(Df^n|E^c_x)d\mu
=\lim_{n\to+\infty}\frac1n L_n(f,\mu).\end{eqnarray*}

%On the one hand, from Fatou's Lemma we have:
%\begin{eqnarray*}
%\Lambda^c(f,\mu)&=&\int\lambda^c(x,f)d\mu(x)\\
%&=&\int\liminf_{n\to\infty}\frac 1n\log {\rm m}(Df^n|E^c_x)d\mu(x)\\
%&\leq &\liminf_{n\to\infty}\frac 1n\int\log  {\rm m}(Df^n|E^c_x)d\mu(x)\\
%&=&\lim\frac 1n L_n(f,\mu)=\alpha.
%\end{eqnarray*}
%On the other hand, taking in account that ${\rm m} (AB)\geq {\rm m}(A){\rm m}(B)$, we have
%\begin{eqnarray*}\lambda^c(x,f)&=&\liminf_{n\to+\infty}\frac{1}{n}\log{\rm m}(Df^n|E^c_x)\\
%&\geq & \liminf_{n\to+\infty} \frac{1}{n} \sum_{j=0}^{n-1}\log {\rm m}(Df|E^c_{f^jx}).
%\end{eqnarray*}
%Integrating  the inequality above and Birkhoff's  theorem applied to $\varphi(\cdot)=\log {\rm m}(Df|E^c_{\cdot})$ allows to conclude that for every Gibbs $u$-state $\mu$,
%$$\Lambda^c(f,\mu)\geq \int\log {\rm m}(Df|E^c_{x})d\mu(x)=L_1(f,\mu).$$ 
%In general, for every integer $n\geq 1$, applying Birkhoff's Theorem to $\varphi(\cdot)=\log{\rm m}(Df^n|E^c_{\cdot})$
%\begin{eqnarray*}L_n(f,\mu)&=&\int\log{\rm m}(Df^n|E^c_x)d\mu(x)\\
%&=&\int \lim_{k\to+\infty}\frac1k\sum_{j=0}^{k-1}\log{\rm m}(Df^{kn}|E^c_{x})d\mu(x)\\
%&=&\int \lim_{k\to+\infty}\frac1k\sum_{j=0}^{k-1}\log{\rm m}(Df^{kn}|E^c_{f^{jn}x})d\mu(x)\\
%& \leq &\int \lim_{k\to+\infty}\frac1k \log{\rm m}(Df^{kn}|E^c_{x})d\mu.
%\end{eqnarray*}
%Then,
%$$\frac{1}{n}L_n(\mu)\leq\int \lim_{k\to+\infty}\frac1{kn}\log{\rm m}(Df^{kn}|E^c_{x})d\mu\leq\int\lambda^c(x,f)d\mu=\Lambda^c(f,\mu).$$
%So we conclude that $\alpha\leq \Lambda^c(f,\mu)$. 
\endproof

\begin{prop}
The function $\Lambda^c\::\:\mathcal{S}\to\mathbb{R}$ is lower semi-continuous.
\end{prop}

\proof Let $(f,\mu)\in\mathcal{S}$ and fix $\epsilon>0$ arbitrarily. From Lemma~\ref{le:Lambdalim}, we can take $n_0\geq 1$ large so that
$$\frac1{n_0}L_{n_0}(f,\mu)>\Lambda^c(f,\mu)-\epsilon.$$
Continuity of $L_{n_0}$ allow us choose thereafter a neighbourhood $\mathcal{N}$ of $(f,\mu)$ in $\mathcal{S}$, small enough for
$$\frac1{n_0}L_{n_0}(g,\mu_g)>\Lambda^c(f,\mu)-\epsilon$$
to hold for any pair $(g,\mu_g)\in\mathcal{N}$. Again from Lemma~\ref{le:Lambdalim}, we have
$$\Lambda^c(g,\mu_g)=\lim \frac{1}{n}L_n(g,\mu_g)=\sup_{n} \frac{1}{n}L_n(g,\mu_g)\geq\frac1{n_0}L_{n_0}(g,\mu_g)\geq\Lambda^c(f,\mu)-\epsilon.$$
%\begin{eqnarray*}
%\Lambda^c(g,\mu_g)&=&\lim_{n\to+\infty}\frac{1}{n}\int\log{\rm m}(Dg^{n}|E^{c}_{g^n(x)})\:d\mu_g(x)\\
%&\geq &\lim_{n\to+\infty}\frac{1}{nn_0}\sum_{j=0}^{n-1}\int\log{\rm m}(Dg^{n_0}|E^{c}_{g^{jn_0}(x)})\:d\mu_g(x)\\
%&\geq &\liminf_{n\to+\infty}\frac{1}{nn_0}\sum_{j=0}^{n-1}n_0(\Lambda^c(f,\mu)-\epsilon)\\
%&=& \Lambda^c(f,\mu)-\epsilon
%\end{eqnarray*}
which proves the proposition.
\endproof

The next proposition corresponds to the statement in Theorem~\ref{mteo:B}.

\begin{prop}\label{prop:meaberto}
 The class of  mostly expanding partially hyperbolic diffeomorphisms constitutes a $C^r$-open subset of ${\rm Diff}^r(M)$, $r>1$.
\end{prop}

\proof Let ${\rm Diff}^r(M)$, $r>1$, be a mostly expanding partially hyperbolic diffeomorphism. Since partial hyperbolicity is a $C^1$-open property (hence, it is $C^r$-open, for every $r\geq 1$), we need to prove the existence of a small neighborhood where every partial hyperbolic diffeomorphism  is mostly expanding.  We argue by contradiction. Assume that there is a sequence $(f_n)_{n\geq1}$ of $C^r$-diffeomorphisms converging to $f$ in the $C^r$-topology, $r>1$, such that, for every $f_n$, $n\geq1$, there is a Gibbs $u$-state $\mu_n$ such that $\Lambda^c(f_n,\mu_n)\leq 0$.

Taking a subsequence if it is necessary, we may assume that $\mu_n$ converges to $\mu$. As already pointed out (see \cite[Remark 11.15]{BDV2005}), $\mu$ is a Gibbs $u$-state for $f$. Since $f$ is  mostly expanding, then $\Lambda^c(f,\mu)>0$. By the lower semicontinuity of $\Lambda^c$, we have
$$0\geq \liminf_{n\to+\infty}\Lambda^c(f_n,\mu_n)\geq\Lambda^c(f,\mu)>0,$$
which is a contradiction.
\endproof

%%%%%%%%%%%%%%%%%%%%%%%%%%%%%%%%%%%%%%%%%%%%%%%%%%%%%%%%%%%%%%%
%%%%%%%%%%%%%%%%% DEMOSTRACION TEOREMA C y D %%%%%%%%%%%%%%%%%%%%%%
%%%%%%%%%%%%%%%%%%%%%%%%%%%%%%%%%%%%%%%%%%%%%%%%%%%%%%%%%%%%%%%

\section[Proof of Theorems~\ref{mteo:C} and~\ref{mteo:D}]{Proof of Theorem~\ref{mteo:C} and Theorem~\ref{mteo:D}}\label{sec:demteosCD}

The way we prove Theorem~\ref{mteo:C} is to prove that if $f$ is  mostly
expanding , then it has an iterate that satisfies the weakly expanding condition (\ref{eq:wNUE}) along the center direction. We do this in several steps. Thereafter we use the weakly expanding condition to show that $f$ has a finite number of physical measures whose basins cover Lebesgue almost every point in the ambient manifold.

\begin{lema}\label{le:Nunif}
Let $f\in{\rm Diff}^r(M)$, $r>1$, be  partially hyperbolic and mostly expanding. Then there exists an integer $n_0\geq 1$ such that

\begin{equation}\label{eq:Nunif}
\int\log{\rm m}(Df^{n_0}|E^c_{x})\: d\mu(x)>0,
\end{equation}
for every Gibbs $u$-state $\mu$ of $f$.
\end{lema}

\proof  For any $n\geq 1$, consider the set
\begin{equation}
G_n=\{\mu\in\mathcal{G}^u(f)\::\: L_n(f,\mu) >0\}.
\end{equation}
Because each  $L_n$ is continuous, $n\geq 1$, we have  $G_n$ is  open in $\mathcal{G}^u(f)$. Since $f$ is  mostly expanding,  it follows that, given any $\mu \in \mathcal{G}^u(f)$, there exists $n \geq 0$ such that $\mu \in G_n$ (see Lemma~\ref{le:equivpositivo} and Lemma~\ref{le:Lambdalim} ). Hence the family $\{G_n\}_{n\in\mathbb{N}}$ is an open covering of $\mathcal{G}^u(f)$ and by compactness,  there exists integers $n_1,\dots, n_k$ such that $\mathcal{G}^u(f)=\cup_{j=1}^k G_{n_j}$.

Let  $n_0:=n_1\cdots n_k$. Since the sequence $(L_n(f,\mu))_n$ is super additive, we have 
\begin{equation}\label{subadditive}
L_{rs}(f,\mu)\geq rL_s(f,\mu), \quad \mbox{ for every integers }r,s\geq 1.
\end{equation}
If $\mu \in \mathcal{G}^u(f)$, then $\mu \in G_{n_j}$ for some $1\leq j \leq k$. Taking $s=n_j$ and $r= n_0/n_j$ in (\ref{subadditive}) proves that $L_{n_0}(\mu)>0$ for every $\mu \in \mathcal{G}^u(f)$.
\endproof

Note that for $n_0 \geq 1$, we always have $\mathcal{G}^u(f) \subseteq \mathcal{G}^u(f^{n_0})$, and it may be that $\mathcal{G}^u(f)$ is a proper subset of $\mathcal{G}^u(f^{n_0})$. Lemma \ref{le:Nunif} states that (\ref{eq:Nunif}) holds for every $\mu \in \mathcal{G}^u(f)$. We do not know whether it holds for every $\mu \in \mathcal{G}^u(f^{n_0})$.

\begin{lema}\label{wNUE on H}
Let $f\in{\rm Diff}^r(M)$, $r>1$, be mostly expanding. Then, there exists an integer $n_0\geq 1$ such that $f^{n_0}$ satisfies the wNUE-condition on a set $H\subseteq M$. 

Moreover, $H \cup f^{-1}(H) \cup \ldots \cup f^{-(n_0-1)}(H)$ has full Lebesgue measure in $M$.
\end{lema}

The proof of Lemma \ref{wNUE on H} makes use of an auxiliary result regarding $\limsup$ of general sequences.

\begin{lema}\label{liminf of average}
Let $\{a_n\}_{n\in\mathbb{N}}$ be a sequence of real numbers and $N\geq 1$ some  integer. Then
\begin{equation}\label{sum over N terms}
\limsup_{n \to \infty} \frac{1}{nN} \sum_{k=0}^{nN-1}  a_{k} 
\leq \max_{0\leq \ell \leq N-1} \limsup_{n \to \infty} \frac{1}{n} \sum_{k=0}^{n-1} a_{kN+\ell}.
\end{equation}
\end{lema}
\proof Let $A=\max_{0\leq \ell \leq N-1} \limsup_{n \to \infty} \frac{1}{n} \sum_{k=0}^{n-1} a_{kN+\ell}$. Assume that $A<+\infty$, otherwise there is nothing to prove. For each $\ell=0,\dots, N-1$, let
$$A_\ell=\limsup_{n \to \infty} \frac{1}{n} \sum_{k=0}^{n-1} a_{kN+\ell}.$$
Fix $\epsilon>0$. For every $\ell=0,\dots, N-1$, there exist $m_\ell\geq 1$ such that 
$$\frac{1}{n} \sum_{k=0}^{n-1} a_{kN+\ell}<A_\ell+\epsilon,\qquad \mbox{ for every } n\geq m_\ell.$$
Let $m:=\max\{m_0,\dots,m_{N-1}\}$. Then 
$$\frac{1}{n} \sum_{k=0}^{n-1} a_{kN+\ell}<A_\ell+\epsilon,\qquad \mbox{ for every } n\geq m,\mbox{ and } \ell=0,\dots N-1.$$
Hence,
$$
\frac{1}{n}\sum_{\ell=0}^{N-1}\sum_{k=0}^{n-1}a_{kN+\ell}<\sum_{\ell=0}^{N-1}A_\ell+\epsilon\leq N(A+\epsilon)
$$
and
$$\frac{1}{nN} \sum_{k=0}^{nN-1}  a_{k}\leq A+\epsilon, \qquad \mbox{for every } n\geq m,$$
which proves the lemma.
\endproof

We need to remark that it is not possible to change $\limsup$ to $\liminf$ in Lemma~\ref{liminf of average}. It is due to this limitation that we are only able to prove \emph{a priori} that an iterate of  $f$, $g=f^{n_0}$ satisfies the wNUE-condition on the set of positive Lebesgue measure $H$. Of course, as consequence of the existence of physical measure for $g=f^{n_0}$, \emph{a posteriori} $g$ satisfies also the NUE-condition on $H$. In fact, it follows from Proposition~\ref{prop:ADLP2014} that $H$ contains the basin of the physical measures of $g$, so we can change the $\limsup$ to $\liminf$ on points belonging to one of such basins. %Nevertheless, we do not know whether  $f$ itself satisfies the wNUE-condition, although we believe it does.

\begin{proof}[Proof of Lemma \ref{wNUE on H}] Recall that if $f$ is partially hyperbolic, we denote by $E^{cu}=E^c\oplus E^u$. Let $n_0$ be as in Lemma \ref{le:Nunif} and write $g = f^{n_0}$. Let
\begin{equation}\label{eq:consthyp}
c_0:=\inf_{\mu \in \mathcal{G}^u(f)} \int \log {\rm m}(Dg \vert E^{cu}) \ d\mu. 
\end{equation}
Then $c_0>0$ according to Lemma \ref{le:Nunif}. Let $E(g)$ be the set of points $x \in M$ such that every accumulation point of the measure
\begin{equation}
\nu_{x,n} = \frac{1}{n} \sum_{k=0}^{n-1}\delta_{g^k(x)}
\end{equation}
belongs to $\mathcal{G}^u(g)$. Then $E(g)$ has full Lebesgue measure in $M$ according to Proposition \ref{prop:uM6}. Note that if $x\in E(g)$ then, $\nu_{x,n}$ accumulates on $\mu\in \mathcal{G}^u(g)$ if and only if $\nu_{f(x),n}$ accumulates on $f_* \mu\in \mathcal{G}^u(g)$. So $E(g)$ is $f$-invariant. Note also that if $\nu \in \mathcal{G}^u(g)$, then 
$$\frac{1}{n_0}(\nu + f_* \nu + \ldots + f_*^{n_0-1}\nu) \in \mathcal{G}^u(f).$$ 

\noindent In particular, every accumulation point $\tilde\mu$ of 
\begin{equation}
\mu_{x,n} = \frac{1}{n_0} \sum_{\ell=0}^{n_0-1} \nu_{f^\ell(x),n}
\end{equation}
is a Gibbs $u$-state for $f$.

Fix some $x \in E(g)$ and for any integer $n\geq 1$ set 
\begin{equation}
a_n = \log {\rm m}(Dg \vert E^{cu}_{f^n(x)}).
\end{equation}
For every $0\leq \ell \leq n_0-1$ fixed we have
\begin{eqnarray*}\label{a_n is equal to integral}
 \int \log {\rm m}(Dg \vert E^{cu}_x)  \ d\nu_{f^\ell (x), n}
 &=&\frac{1}{n} \sum_{k=0}^{n-1} \log {\rm m}(Dg \vert E^{cu}_{g^{k}(f^\ell (x))})\\
 &=&\frac{1}{n} \sum_{k=0}^{n-1} \log {\rm m}(Dg \vert E^{cu}_{f^{kn_0+\ell}(x)})\\
 &=&\frac{1}{n} \sum_{k=0}^{n-1} a_{kn_0+\ell}.
\end{eqnarray*}
And so
\begin{eqnarray*}
\int \log {\rm m}(Dg \vert E^{cu}_x)  \ d\mu_{x,n}&=&\frac{1}{n_0}\sum_{\ell=0}^{n_0-1}\int \log {\rm m}(Dg \vert E^{cu}_x)\ d\nu_{f^\ell (x), n}\\ &=&\frac{1}{n_0}\sum_{\ell=0}^{n_0-1}\frac{1}{n} \sum_{k=0}^{n-1} a_{kn_0+\ell}. 
\end{eqnarray*}

For every sufficiently large $n\geq 1$ we must have 
\begin{eqnarray*}
\frac{1}{nn_0}\sum_{\ell=0}^{n_0-1} \sum_{k=0}^{n-1} a_{kn_0+\ell} &=& \int \log {\rm m}(Dg \vert E_x^{cu}) \ d\mu_{x,n}\\
&>& \frac 12 \inf_{\mu \in \mathcal{G}^u(f)} \int \log {\rm m}(Dg \vert E^{cu})d\mu  = c_0/2.
\end{eqnarray*}
%then $\mu_{x,n}$ would have to accumulate on some $\tilde{\mu} \in \mathcal{G}^u(f)$ with $$\int \log m(Dg| E^{cu}) \: d\tilde{\mu} \leq c_0/2,$$ which is absurd.
In particular,  
\begin{equation}
\limsup_{n \to \infty} \frac{1}{nn_0}\sum_{k=0}^{n-1}\sum_{\ell=0}^{n_0-1} a_{kn_0+\ell}
\geq c_0/2.
\end{equation}
Hence, applying Lemma \ref{liminf of average} with $N=n_0$, we conclude that there exists some $0\leq \ell \leq n_0-1$ such that
\begin{equation}
\limsup_{n \to \infty} \frac{1}{n} \sum_{k=0}^{n-1} a_{kn_0+\ell} \geq c_0/2.
\end{equation} 
We have now proved that $f^\ell(x) \in H$ for some $0 \leq \ell n_0-1$. Since $x \in E$ was chosen arbitrarily, this implies that $E = H \cup f^{-1}(H) \cup \ldots \cup f^{-(n_0-1)}(H)$. To finish the proof, note that   $H$  corresponds to the set of points $x\in M$ for which
$$\limsup_{n \to \infty} \frac{1}{n} \sum_{k=0}^{n-1} \log {\rm m}(Dg \vert E_{g^k(x)}^{cu}) =\limsup_{n \to \infty} \frac{1}{n} \sum_{k=0}^{n-1} \log {\rm m}(Dg \vert E_{f^{n_0k}(x)}^{cu}) \geq c_0/2>0.$$ 
which is equivalent to
\begin{equation}
\liminf_{n \to \infty} \frac{1}{n} \sum_{k=0}^{n-1} \log \| (Dg \vert E_{g^k(x)}^{cu})^{-1}\| \leq -c_0/2<0.
\end{equation}
so $g=f^{n_0}$ satisfies the wNUE-condition on $H$.
\end{proof}

Next lemma allows us to  conclude Theorem~\ref{mteo:C}.

\begin{lema}\label{le:existecumedida} 
Let $f\in{\rm Diff}^r(M)$, $r>1$, be partially hyperbolic and  mostly expanding. Then there
exist finitely many ergodic Gibbs $cu$-states $\mu_1,\dots\mu_\ell$  all of which are physical measures. The union of their basins have full Lebesgue measure in $M$.
\end{lema} 

\proof  From Lemma~\ref{wNUE on H}, there exists an integer $n_0\geq 1$ such that $f^{n_0}$ satisfies the wNUE-condition on a set $H\subseteq M$. So we can apply Proposition~\ref{prop:ADLP2014} to $g=f^{n_0}$ and then we conclude that exist finitely many ergodic Gibbs $cu$-states  (of $g$) 
$\nu_1, ..., \nu_{\ell}$, whose basins have nonempty interior and such that 

$${\rm Leb}(H\setminus\cup_{j=1}^\ell\mathcal{B}(\nu_j))=0.$$

As we noted above, if  $\nu_j$ is a Gibbs $cu$-state for $g=f^{n_0}$, $j\in\{1,\dots,\ell\}$, then 
$$\mu_j:=\frac{1}{n_0}(\nu_j + f_* \nu + \ldots + f_*^{n_0-1}\nu_j)$$ 
is a Gibbs $cu$-state for $f$. Of course each $ \mathcal{B}(\mu_j)$ is contained in $H \cup f^{-1}(H) \cup \ldots \cup f^{-(n_0-1)}(H)$ and
$${\rm Leb}\left(\cup_{k=0}^{n_0-1}f^{-k}(H)\setminus\cup_{j=1}^\ell  \mathcal{B}(\mu_j)\right)=0.$$
Since $\cup_{k=0}^{n_0-1}f^{-k}(H)$ has full Lebesgue measure in $M$ (see Lemma~\ref{wNUE on H}), then almost  every $x \in M$ we have $x\in \mathcal{B}(\mu_j)$ for some $1 \leq j \leq \ell$.

\endproof

The following remark is a key ingredient in the proof of Theorem~\ref{mteo:D}.

\begin{lema}\label{le:openbasin} 
Let $f\in{\rm Diff}^r(M)$, $r>1$, be partially hyperbolic and mostly expanding. Let $\mu$  be a physical measures for $f$. Then its basin is open in $M$ up to a zero Lebesgue measure subset.
\end{lema} 

\proof We prove the Lemma by completeness, even though it is contained in the second statement of Proposition~\ref{prop:ADLP2014} (see also \cite{ADLP2014}). Since $\mu$ is an ergodic, $\mu$-almost every point belongs to $\mathcal{B}(\mu)$. Moreover, since $\mu$ is a  Gibbs $cu$-state, $\mu$-almost every point lies in an unstable leaf $F$ on which ${\rm Leb}_F$-almost every point does also belong to $\mathcal{B}(\mu)$. More precisely, there exist a set $\mathcal{L}(\mu)$ of unstable leaves such that 
\begin{enumerate}
\item $\displaystyle \mu\left(\bigcup_{F\in \mathcal{L}(\mu)}F\right)=1.$
\item $\mathcal{B}(\mu)$ has full ${\rm Leb}_F$-measure for every $F\in  \mathcal{L}(\mu)$.
\end{enumerate}
For each $F \in \mathcal{L}(\mu)$, let $\mathcal{A}(F)$ be the $s$-saturated set  consisting of the union of all strong stable manifolds of points in $F$. Then $\mathcal{A}(F)$ is an open set and, by absolute continuity of the stable foliation, $\mathcal{B}(\mu)\cap \mathcal{A}(F)$ has full Lebesgue measure in $\mathcal{A}(F)$. Let $\mathcal{A} = \bigcup_{F \in \mathcal{L}(\mu)} \mathcal{A}(F)$.  Since $\mathcal{B}(\mu)$ has full Lebesgue measure in each $\mathcal{A}(F)$, it has full Lebesgue measure in $\mathcal{A}$. Note that $\mathcal{A}$ is a neighborhood of $\operatorname{supp} \mu$. Therefore, given  any point $x\in \mathcal{B}(\mu)$, there is an iterate $f^n(x)$ that belongs to $\mathcal{A}$. But $\mathcal{A}$ is invariant, so in fact we must have $x \in \mathcal{A}$. Hence $\mathcal{A}$ is an open set containing $\mathcal{B}(\mu)$, and $\mathcal{B}(\mu)$ has full Lebesgue measure in $\mathcal{A}$.

\endproof

Next Lemma implies Theorem~\ref{mteo:D}.

\begin{lema}\label{le:unique} 
Let $f\in{\rm Diff}^r(M)$, $r>1$, be  transitive, partially hyperbolic, and mostly expanding. Then $f$ has a unique  physical measure  whose basin has full Lebesgue measure in $M$.
\end{lema}

\proof Assume that there are two physical measures $\mu$ and $\nu$ for $f$. The basins $\mathcal{B}(\mu)$ and $\mathcal{B}(\nu)$ are open up to a zero Lebesgue measure subset. From topological transitivity, there is a non-negative integer $n$ such that $f^n(\mathcal{B}(\mu))\cap\mathcal{B}(\nu)\ne\emptyset$. Hence $\mu=\nu$.
\endproof

%
%%%%%%%%%%%%%%%%%%%%%%%%%%%%%%%%%%%%%%%%%%%%%%%%%%%%%%%%%%%%%%%
%%%%%%%%%%%%%%%%%         EJEMPLOS       %%%%%%%%%%%%%%%%%%%%%%
%%%%%%%%%%%%%%%%%%%%%%%%%%%%%%%%%%%%%%%%%%%%%%%%%%%%%%%%%%%%%%%

\section{Dolgopyat-Hu-Pesin blocks}\label{sec:DHP Blocks}
 
In \cite[Appendix B]{BP1999}, Dolgopyat, Hu and Pesin provide an example of a non-uniformly hyperbolic volume preserving diffeomorphism on $\mathbb{T}^3$ with countably many ergodic components. As a key step in the construction they consider a linear Anosov diffeomorphism $A$ on $\mathbb{T}^2$ and the map $F\::\:[0,1] \times \mathbb{T}^2\to[0,1] \times \mathbb{T}^2$ defined by $F=I\times A$, where $I$ is the identity map. Then, by a suitable perturbation of $F$, they construct a $C^{\infty}$ diffeomorphism $g$ on $M=[0,1] \times \mathbb{T}^2$, with the following properties:
\begin{itemize}
\item $g$ is a partially hyperbolic  volume preserving diffeomorphism on $M$.
\item The connected components of the boundary, $W_i := \{i \} \times \mathbb{T}^2 , \ i = 0,1$ are left invariant by $g$ and $g|W_i=F|W_i$ is a linear Anosov map.
\item $g$ is ergodic with respect to Lebesgue measure and has positive central Lyapunov exponents almost everywhere.
\item $g$ can be chosen to be as $C^r$ close to $F$ as desired, for any $r \geq 2$.
\end{itemize}

The meaning of the first item above, i.e. that $g$ is partially hyperbolic, could potentially be confusing, as we are dealing with a manifold with boundary, and no definition of partial hyperbolicity has been provided in this case. In fact there is no need for that. The diffeomorphism $g$ is constructed so that it becomes a $C^{\infty}$ partially hyperbolic diffeomorphism on $\mathbb{T}^3$ when identifying $\{0\}\times \mathbb{T}^2$ with $\{1\}\times \mathbb{T}^2$ in the natural manner. This is what enables them to glue two or more copies of $g$ together.

We observe that if $\mu$ is a Gibbs $u$-state for $A$, then $\delta_0\times\mu$ is a $u$-measure for $g$ with zero central Lyapunov exponent, where $\delta_0$ is the Dirac measure concentrate on $0$. Therefore $g$ is not mostly expanding, although it does satisfy the non uniformly expanding along the center (or center-unstable) direction condition \eqref{eq:NUE}. 

Although the construction of the diffeomorphism $g$ by Dolgopyat, Hu, and Pesin was done as an intermediate step in providing an example of a diffeomorphism with non-zero Lyapunov exponents almost everywhere and yet having countable many ergodic components, it turns out to useful when thinking about the NUE-condition of \cite{ABV2000} and how it compares to our notion of mostly expanding central direction. 

Gluing copies of $g$ together is an easy matter.
Indeed, given $0<\lambda<1$ and $0 \leq \tau < 1- \lambda$, define the squeezing and sliding action
\begin{eqnarray*} 
 L_{\lambda, \tau}: [0,1) \times \mathbb{T}^2 &\rightarrow & [\tau, \tau + \lambda) \times \mathbb{T}^2 \\
       (x,y,z) & \mapsto & (\lambda x + \tau, y, z).
\end{eqnarray*}

Now suppose we wish to construct an example of a  diffeomorphism $f: \mathbb{T}^3 \rightarrow \mathbb{T}^3$ satisfying the NUE-condition and exhibiting  precisely $k$ physical measures, all we have to do is to identify $\mathbb{T}^3$ with $[0, 1) \times \mathbb{T}^2$ in the obvious way and take $f$ to be the diffeomorphism satisfying 

\[f\vert_{ [ \frac{i}{k}, \frac{i+1}{k})}= L_{\frac{1}{k}, \frac{i}{k}} \circ g \circ L_{\frac{1}{k}, \frac{i}{k}}^{-1}, \ i=0, \ldots, k.\]

\begin{prop} Having condition (\ref{eq:NUE}) or (\ref{eq:wNUE}) satisfied on a set of full Lebesgue measure is not a robust property.
\end{prop}

\proof Pick any $0<\epsilon<1$ and define $f_{\epsilon}$ by 
\begin{align*}
 f_{\epsilon} (x,y,z) & =  L_{1-\epsilon, 0} \circ g \circ L_{1-\epsilon, 0}^{-1} (x,y,z), \quad & x \in [0, 1-\epsilon) \\
  f_{\epsilon} (x,y,z) & =  (x, A(y,z)),  &x \in [1-\epsilon, 1).
\end{align*}
By construction, $f_{\epsilon}$ approaches $g$ as $\epsilon \rightarrow 0$ in any $C^r$ topology, but none of the $f_{\epsilon}$ satisfy the NUE-condition.
\endproof

\begin{prop} Having condition (\ref{eq:NUE}) or (\ref{eq:wNUE}) satisfied on a set of positive Lebesgue measure is not a robust property.
\end{prop}
\proof Let $f$ be a diffeomorphism on $\mathbb{T}^3$ obtained by gluing two blocks as above with different sign on their central Lyapunov exponents. More precisely, let $f$ be such that 
$$
f\vert_{ [0,\frac{1}{3}) \times \mathbb{T}^2 }= L_{\frac{1}{3}, 0} \circ g \circ L^{-1}_{\frac{1}{3}, 0},
$$
and
$$
f\vert_{  [\frac{1}{3},1) \times \mathbb{T}^2 } = L_{\frac{2}{3}, \frac{1}{3}} \circ g^{-1} \circ L^{-1}_{\frac{2}{3}, \frac{1}{3}},
$$

Then both $f$ satisfies (\ref{eq:NUE}) and (\ref{eq:wNUE}) on a set of positive measure. Note that the integrated central Lyapunov exponent of $f$ is negative. From \cite{RHRHU2008} we know that, for every $r>1$, there is a $C^r$ diffeomorphism $\tilde{f}$ arbitrarily close to $f$ in the $C^r$ topology, such that $\tilde{f}$ is ergodic. In particular the central Lyapunov exponent of $\tilde{f}$ is constant Lebesgue almost everywhere. Since $\tilde{f}$ is $C^1$ close to $f$, this Lyapunov exponent is negative. In particular, $\tilde{f}$ fails to satisfy conditions (\ref{eq:NUE}) and (\ref{eq:wNUE}) on a set of positive Lebesgue measure.\endproof

\begin{prop} Having condition (\ref{eq:NUE}) or (\ref{eq:wNUE}) satisfied on a set of full measure does not imply that the number of physical measures vary upper-semi-continuously with the dynamics.
\end{prop}

\proof To prove this  claim, we have to show that the number of physical measures can undergo an explotion. All we have to do to accomplish that is to modify $f_{\epsilon}$ in Claim 1 on $[1-\epsilon, 1) \times \mathbb{T}^2$. That is,
\begin{align}
 f_{\epsilon} (x,y,z) & =  L_{1-\epsilon, 0} \circ g \circ L_{1-\epsilon, 0}^{-1} (x,y,z), \quad & x \in [0, 1-\epsilon) \\
  f_{\epsilon} (x,y,z) & =  L_{\epsilon, 1-\epsilon} \circ g' \circ L_{\epsilon, 1-\epsilon}^{-1} (x,y,z),  &x \in [1-\epsilon, 1),
\end{align}
where $g':[0,1]\times \mathbb{T}^2$ is a diffeomorphism satisfying all the properties of $g$, and sufficiently close to  $F$ in the $C^r$ topology. The resulting $f_{\epsilon}$ has two physical measures and is arbitrarily close to $g$ (which has one).
\endproof

\section{Derived from Anosov  example}\label{sec:DA} Here we show that a classical example of a non-hyperbolic robustly transitive diffeomorphism due to Ma\~{n}e \cite{M1978} satisfies our notion of partial hyperbolicity with mostly expanding central direction in the strong sense. We do that following the ideas developed in \cite{BV2000}.

We start with   a linear Anosov diffeomorphism $f_0\::\:\mathbb{T}^3\to \mathbb{T}^3$ with three different eigenvalues 
$$\lambda_s<1/3<1<\lambda_c\ll3<\lambda_u.$$ 
We consider the splitting 
$$T\mathbb{T}^3=E^s_0\oplus E^c_0\oplus E^u_0,$$
into eigenspaces corresponding to these eigenvalues. Let $p_0\in\mathbb{T}^3$ be a fixed point of $f_0$ and consider $\delta>0$ (to be fixed later). We deform $f_0$ along a one-parameter family of diffeomorphisms $f_t$, $t\geq 0$, by isotopy inside $V=B(p_0,\delta)$, to make it  go through a pitchfork bifurcation. The expansion in the strong unstable subbundle $E^u_t$ remains large everywhere and the same is true for the contraction in the strong-stable subbundle $E^s_t$. More precisely, for every parameter $t\geq 0$,
$$
|Df_t|E^s_t|<1/3 <3< |Df_t|E^u_t|.
$$
Then, the distortion along the strong unstable leaves remains uniformly bounded in the whole family $f_t$. 

The center subbundle $E^c_t$ restricted outside $V$ also remains expanding, that is, 
\begin{equation}\label{eq:expcen}
3>|Df_t(x)|E^c_t(x)|\geq \eta_c(t)>1, \text{ for every }  x\in\mathbb{T}^3\setminus V.
\end{equation} 
where $\eta_c(0)=\lambda_c>1$. Moreover, for every parameter $0\leq t<t_0$, the deformation can be done  in a such a way  that \eqref{eq:expcen} can be extended to every $x\in \mathbb{T}^3$ choosing carefully the constant $\eta_c(t)>1$ and so,  $f_t$ is expanding along the central direction $E^c_t$. If we denote by $p_t$, $t\geq 0$, the continuation of the hyperbolic fixed point $p_0$ of $f_0$ inside the neighborhood $V$, the eigenvalue $\lambda_c(p_t)>1$ if $0\leq t<t_0$ and becomes $\lambda_c(p_{t_0})=1$. Then, for parameters $t_0+\epsilon>t>t_0$, sightly greater than $t_0$,  the continuation $p_t$ is a saddle point whose stable index changes from 1 to 2, and two other saddle points $q_t^1, q_t^2\in V$, of stable index 1, are created (See Figure~\ref{figura:DA}).

\begin{figure}
\begin{center}
\includegraphics[scale=0.7]{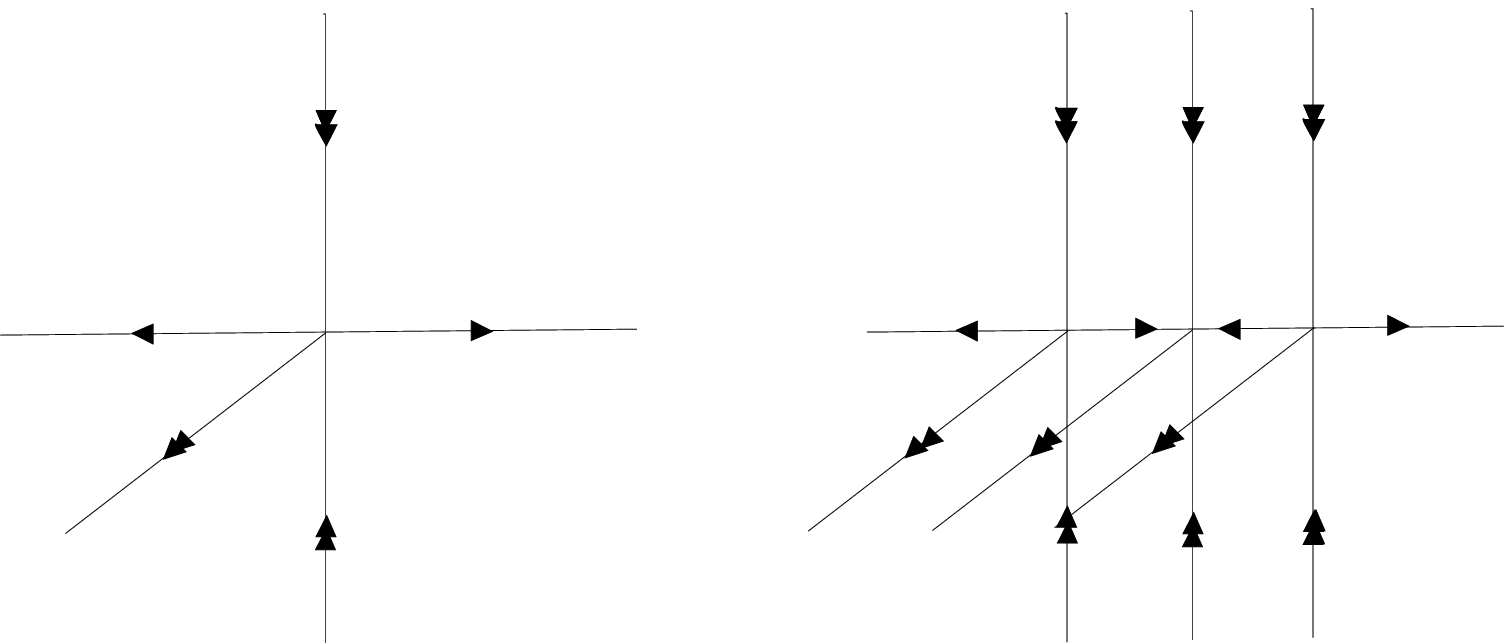} 
\end{center}
\caption{}
\label{figura:DA}
\end{figure}

Note that  for every $\beta>0$, we can choose $\epsilon>0$ such that for every $0\leq t< t_0+\epsilon$ and every $x\in V$,
$$|Df_t(x)|E^c_t(x)|>1-\beta.$$

Then, for $0\leq t< t_0+\epsilon$, the diffeomorphims $f_t$ is partially hyperbolic. Hence there exist unique foliations $\mathcal{F}^s_t$, $\mathcal{F}^u_t$ tangent to $E^s_t$, $E^u_t$ respectively. Moreover,  it follows from \cite{HPS1977}    that $f_t$ has an invariant central foliation $\mathcal{F}^c_t$ tangent to the central direction $E_t^c$, $0\leq t< t_0+\epsilon$. Since  it remains normally contracting all the way during the isotopy the center unstable foliation $\mathcal{F}^{cu}_t$ is leaf conjugate to the unstable foliation $\mathcal{F}^{cu}_0$, that means there exists a homeomorphism $h:\mathbb{T}^3\to \mathbb{T}^3$ which sends leaves of $\mathcal{F}^{cu}_t$ into leaves of $\mathcal{F}^{cu}_0$ and conjugates the dynamics of the leaves. More precisely,
$$h(W^{cu}_t( f_t(x))) = W^{cu}_0 ( f_0(h(x))).$$

If we fix a diffeomorphism $f_t$ with $t$ sightly greater than $t_0$ in the family above, we obtain a $C^1$-open set $\mathcal{U}$ of diffeomporphisms containing $f_t$, such that  every $f\in\mathcal{U}$ satisfies:

\begin{enumerate}
\item $f$ is partially hyperbolic and uniquely integrable. We have an $Df$-invariant splitting into three non trivial subbundles
$$T\mathbb{T}^3=E^s\oplus E^c\oplus E^u,$$
satisfying
\begin{equation}\label{eq:DAunif}
|Df|E^s|<1/3 <3< |Df|E^u|.
\end{equation}
and for every $x\in V$,
\begin{equation}\label{eq:DAcenter}
|Df(x)|E^c(x)|>1-\beta.
\end{equation}

\item There exist unique foliations $\mathcal{F}^*$ tangent to $E^*$ respectively, where $*=s,c,u$.

\item $f$ has three hyperbolic fixed saddles inside $V$ contained in a same central leaf $\mathcal{F}^c(p)$: two saddle with stable index 1 and one saddle with stable index 2.
\end{enumerate}

Recall that  for a $u$-segment $\gamma$, by \eqref{eq:DAunif} we have
$${\rm length}(f(\gamma))\geq 3\cdot{\rm length}(\gamma),$$
so we can assume that there are constants $L\geq 0$ such that, given any $u$-segment $\gamma$ with
$${\rm length}(\gamma)\geq L,$$
then there is an integer $k=k(\gamma)\geq 1$ such that we may write
$$f(\gamma)=\gamma_1\cup\dots\cup\gamma_k,$$
as the disjoint union of $u$-segments $\gamma_i$  satisfying
$$2L\geq {\rm length}(\gamma_i)\geq L,\quad i=1,\dots,k.$$

By redefining $V$ if it is necessary, we can assume that there is  $0<\tau_0<1$ such that, given any $u$-segment $\gamma$ with 
${\rm length}(\gamma)\geq L$, if $I_V=\{i\::\: \gamma_i\cap V\ne \emptyset,\: i=1,\dots, k(\gamma)\}$, then
$$\sum_{j\in I_V}{\rm length}(\gamma_j)\leq \tau_0\cdot {\rm length}(f(\gamma)).$$
Of course $L\geq 0$ can be chosen close the size of $V$, so the image of any $u$-segment with length bigger than $L$ spend a positive fraction $1-\tau_0$ (in length) outside $V$.

For any integer $k\geq1$, and $0<\alpha<1$, we define

$$M(k,\alpha)=\{x\in\gamma\::\: \frac1k\sum_{j=0}^{k-1}\mathbf{1}_V(f^j(x))\geq\alpha\},$$
where $\mathbf{1}_V$ is the characteristic function of $V$. Note that $x\in M(k,\alpha)$ if its orbit  spends a fraction of time bigger than $\alpha>0$ inside $V$ (until time $k$). Of course, the complement of the set
$$\bigcap_{n\geq 1}\bigcup_{k\geq n}M(k,\alpha),$$
corresponds to the set of points $x\in \gamma$ such that spends at least a fraction $1-\alpha$ outside $V$.

\begin{lema}\label{le:tiempouseg} There is $0<\alpha_0<1$, depending on $\alpha_0=\alpha_0(\tau_0,|Df|E^u|)$ such that, for every $u$-segment $\gamma$, and every $\alpha_0\leq\alpha<1$, the set
$$M_\gamma:=\bigcap_{n\geq 1}\bigcup_{k\geq n}M(k,\alpha),$$
has zero Lebesgue measure in $\gamma$.
\end{lema}
The proof of Lemma~\ref{le:tiempouseg} we can found in \cite{BV2000}, section 6.3. We need to remark that $\alpha _0$ is a constant such that it essentially depends on the distortion bound along the unstable direction. That allow us to fix $\alpha>\alpha_0$ and to choose $\beta$ close to zero such that 

\begin{equation}\label{eq:relclave}
3^{(1-\alpha)}(1-\beta)^\alpha=\lambda>1. 
\end{equation}

\begin{prop} Let $f\in \mathcal{U}$ be as above. Then for any $u$-segment $\gamma$ and Lebesgue almost every point  $x\in \gamma$, 
$$\limsup_{n\to\infty}\frac1n\log|(Df^n|E^c_x)^{-1}|<0.$$
In particular $f$ is mostly expanding.
\end{prop}

\begin{proof} Let $\gamma$ be a $u$-segment and consider $x\in\gamma$. Since $E^c$ is a one dimensional subbundle,

$$|Df^n(x)|E^c(x)|=\prod_{j=0}^{n-1}|Df(f^j(x))|E^c(f^j(x))|.$$
If we denote by $J_V(x)=\{j\::\: f^j(x)\in V, 0\leq j\leq n-1\}$ the iterates of $x$ belonging to $V$ then,  by \eqref{eq:DAcenter}, the derivative of $f^j(x)$ along the center direction is controlled if $j\in J_V(x)$ by
$$|Df(f^j(x))|E^c(f^j(x))|>1-\beta.$$
For the iterated $j\notin J_V(x)$, by \eqref{eq:DAunif}, we have that
$$|Df(f^j(x))|E^c(f^j(x))|>3.$$
Then, 
$$|Df^n(x)|E^c(x)|>3^{n-|J_V(x)|}(1-\beta)^{|J_V(x)|}.$$
Follows from the Lemma~\ref{le:tiempouseg}  that for Lebesgue almost every $x\in\gamma\setminus M_\gamma$, that is, there exist  $N\geq 1$, such that for every $n\geq N$ we have
$$\frac{|J_V(x)|}{n}\leq \alpha<1.$$
Then, taking in account \eqref{eq:relclave}, we have that
\begin{eqnarray*}
|Df^n(x)|E^c(x)|&>&3^{n-|J_V(x)|}(1-\beta)^{|J_V(x)|}\\
&>&3^{(1-\alpha) n}(1-\beta)^{\alpha n}\\
&>&\lambda^n>1.
\end{eqnarray*}

Finally,
$$\frac{1}{n}\log|Df^n(x)|E^c(x)|>\log\lambda>0.$$

and then, we obtain
$$\liminf_{n\to\infty}\frac{1}{n}\log|Df^n(x)|E^c(x)|>\log\lambda>0.$$

\end{proof}

\noindent \emph{Remark:} The previous example can be generalized to $\mathbb{T}^{n+2}$, $n\geq 1$, beginning from a linear Anosov difeomorphism $f_0$ with decomposition 
$$T\mathbb{T}^3=E^s_0\oplus E^c_0\oplus E^u_0,$$
where $\dim E^u_0=\dim E^c_0=1$ and $\dim E^s_0= n$.

\vspace{1 cm}

\noindent {\bf Acknowledgment:} We would like to thank for the referee for the helpful suggestions which we have incorporated into this version.

\bibliographystyle{plain}

\bibliography{MACABIB}

\def\cprime{$'$}
\begin{thebibliography}{10}

\bibitem{ABV2000}
Jos{\'e}~F. Alves, Christian Bonatti, and Marcelo Viana.
\newblock S{RB} measures for partially hyperbolic systems whose central
  direction is mostly expanding.
\newblock {\em Invent. Math.}, 140(2):351--398, 2000.

\bibitem{ADLP2014}
Jose~F. Alves, C.~L. Dias, S.~Luzzatto, and V.~Pinheiro.
\newblock Srb measures for partially hyperbolic systems whose central direction
  is weakly expanding.
\newblock arXiv:1403.2937, 2014.

\bibitem{A2010}
Martin Andersson.
\newblock Robust ergodic properties in partially hyperbolic dynamics.
\newblock {\em Trans. Amer. Math. Soc.}, 362(4):1831--1867, 2010.

\bibitem{BP1999}
L.~Barreira and Ya. Pesin.
\newblock Lectures on {L}yapunov exponents and smooth ergodic theory.
\newblock In {\em Smooth ergodic theory and its applications ({S}eattle, {WA},
  1999)}, volume~69 of {\em Proc. Sympos. Pure Math.}, pages 3--106. Amer.
  Math. Soc., Providence, RI, 2001.
\newblock Appendix A by M. Brin and Appendix B by D. Dolgopyat, H. Hu and
  Pesin.

\bibitem{BDU2002}
Christian Bonatti, Lorenzo~J. D{\'{\i}}az, and Ra{\'u}l Ures.
\newblock Minimality of strong stable and unstable foliations for partially
  hyperbolic diffeomorphisms.
\newblock {\em J. Inst. Math. Jussieu}, 1(4):513--541, 2002.

\bibitem{BDV2005}
Christian Bonatti, Lorenzo~J. D{\'i}az, and Marcelo Viana.
\newblock {\em Dynamics beyond uniform hyperbolicity}, volume 102 of {\em
  Encyclopaedia of Mathematical Sciences}.
\newblock Springer-Verlag, Berlin, 2005.
\newblock A global geometric and probabilistic perspective, Mathematical
  Physics, III.

\bibitem{BV2000}
Christian Bonatti and Marcelo Viana.
\newblock S{RB} measures for partially hyperbolic systems whose central
  direction is mostly contracting.
\newblock {\em Israel J. Math.}, 115:157--193, 2000.

\bibitem{B2015}
R.~T. {Bortolotti}.
\newblock {Physical Measures for Certain Partially Hyperbolic Attractors on
  3-Manifolds}.
\newblock {\em ArXiv e-prints}, June 2015.

\bibitem{BDP2002}
K.~Burns, D.~Dolgopyat, and Ya. Pesin.
\newblock Partial hyperbolicity, {L}yapunov exponents and stable ergodicity.
\newblock {\em J. Statist. Phys.}, 108(5-6):927--942, 2002.
\newblock Dedicated to David Ruelle and Yasha Sinai on the occasion of their
  65th birthdays.

\bibitem{BDPP2008}
Keith Burns, Dmitry Dolgopyat, Yakov Pesin, and Mark Pollicott.
\newblock Stable ergodicity for partially hyperbolic attractors with negative
  central exponents.
\newblock {\em J. Mod. Dyn.}, 2(1):63--81, 2008.

\bibitem{D2000}
Dmitry Dolgopyat.
\newblock On dynamics of mostly contracting diffeomorphisms.
\newblock {\em Comm. Math. Phys.}, 213(1):181--201, 2000.

\bibitem{DVY2016}
Dmitry Dolgopyat, Marcelo Viana, and Jiagang Yang.
\newblock Geometric and measure-theoretical structures of maps with mostly
  contracting center.
\newblock {\em Comm. Math. Phys.}, 341(3):991--1014, 2016.

\bibitem{HP2006}
Boris Hasselblatt and Yakov Pesin.
\newblock Partially hyperbolic dynamical systems.
\newblock In {\em Handbook of dynamical systems. Vol. 1B}, pages 1--55.
  Elsevier B. V., Amsterdam, 2006.

\bibitem{HPS1977}
M.~W. Hirsch, C.~C. Pugh, and M.~Shub.
\newblock {\em Invariant manifolds}.
\newblock Springer-Verlag, Berlin, 1977.
\newblock Lecture Notes in Mathematics, Vol. 583.

\bibitem{M1978}
Ricardo Ma{\~{n}}{\'e}.
\newblock Contributions to the stability conjecture.
\newblock {\em Topology}, 17(4):383--396, 1978.

\bibitem{M1987}
Ricardo Ma{\~n}{\'e}.
\newblock {\em Ergodic theory and differentiable dynamics}.
\newblock Springer-Verlag, Berlin, 1987.

\bibitem{N2015}
Felipe Nobili.
\newblock Minimality of invariant laminations for partially hyperbolic
  attractors.
\newblock {\em Nonlinearity}, 28(6):1897--1918, 2015.

\bibitem{PS1982}
Ya.~B. Pesin and Ya.~G. Sina{\u\i}.
\newblock Gibbs measures for partially hyperbolic attractors.
\newblock {\em Ergodic Theory Dynam. Systems}, (3):417--438 (1983), 1982.

\bibitem{P2010}
Yakov Pesin.
\newblock On the work of {D}olgopyat on partial and nonuniform hyperbolicity.
\newblock {\em J. Mod. Dyn.}, 4(2):227--241, 2010.

\bibitem{PC2010}
Yakov Pesin and Vaughn Climenhaga.
\newblock Open problems in the theory of non-uniform hyperbolicity.
\newblock {\em Discrete Contin. Dyn. Syst.}, 27(2):589--607, 2010.

\bibitem{PS1989}
C.~Pugh and M.~Shub.
\newblock Ergodic attractors.
\newblock {\em Trans. Amer. Math. Soc.}, 312(1):1--54, 1989.

\bibitem{PS2006}
Enrique~R. Pujals and Mart{\'{\i}}n Sambarino.
\newblock A sufficient condition for robustly minimal foliations.
\newblock {\em Ergodic Theory Dynam. Systems}, 26(1):281--289, 2006.

\bibitem{RHRHU2008}
F.~Rodriguez~Hertz, M.~A. Rodriguez~Hertz, and R.~Ures.
\newblock Accessibility and stable ergodicity for partially hyperbolic
  diffeomorphisms with 1{D}-center bundle.
\newblock {\em Invent. Math.}, 172(2):353--381, 2008.

\bibitem{T2005}
Masato Tsujii.
\newblock Physical measures for partially hyperbolic surface endomorphisms.
\newblock {\em Acta Math.}, 194(1):37--132, 2005.

\bibitem{UV2015}
R.~{Ures} and C.~H. {V{\'a}squez}.
\newblock {On the non-robustness of intermingled basins}.
\newblock {\em ArXiv e-prints}, March 2015.

\bibitem{V2007}
Carlos~H. V{\'a}squez.
\newblock Statistical stability for diffeomorphisms with dominated splitting.
\newblock {\em Ergodic Theory Dynam. Systems}, 27(1):253--283, 2007.

\bibitem{V2009}
Carlos~H. V{\'a}squez.
\newblock Stable ergodicity for partially hyperbolic attractors with positive
  central {L}yapunov exponents.
\newblock {\em J. Mod. Dyn.}, 3(2):233--251, 2009.

\bibitem{VY2013}
Marcelo Viana and Jiagang Yang.
\newblock Physical measures and absolute continuity for one-dimensional center
  direction.
\newblock {\em Ann. Inst. H. Poincar\'e Anal. Non Lin\'eaire}, 30(5):845--877,
  2013.

\bibitem{Y2002}
Lai-Sang Young.
\newblock What are {SRB} measures, and which dynamical systems have them?
\newblock {\em J. Statist. Phys.}, 108(5-6):733--754, 2002.
\newblock Dedicated to David Ruelle and Yasha Sinai on the occasion of their
  65th birthdays.

\end{thebibliography}

\end{document}